\documentclass[10pt]{amsart}

\usepackage{amssymb,amsthm,amstext,amsfonts,amsmath,yfonts,latexsym}
\usepackage{xypic}

\newtheoremstyle{theorem}
     {11pt}
     {11pt}
     {}
     {}
     {\bfseries}
     {}
     {.5em}
     {\noindent\thmnumber{#2}. \thmname{#1}\thmnote{#3}}

\theoremstyle{theorem}

\newcommand{\csf}{{}\sp\omega{2}}
\newcommand{\csfo}{\omega\times{}\sp\omega{2}}
\newcommand{\cl}[2][X]{\mathrm{cl}_{#1}({#2})}
\newcommand{\bd}[2][X]{\mathrm{bd}_{#1}({#2})}
\newcommand{\F}{\mathcal{F}}
\newcommand{\meager}{\mathcal{M}}
\newcommand{\covM}{\mathrm{cov}(\mathcal{M})}
\newcommand{\cofM}{\mathrm{cof}(\mathcal{M})}
\newcommand{\Pint}{\mathfrak{p}}
\newcommand{\C}{\mathfrak{c}}
\newcommand{\class}{\mathbf{K}}
\newcommand{\poset}{\mathbb{P}}
\newcommand{\remcard}{\textfrak{re}}
\newcommand{\dg}{\textfrak{dg}}
\newcommand{\U}{\mathcal{U}}

\newcommand{\G}{\mathcal{G}}
\newcommand{\MA}[2][\kappa]{\mathbf{MA}_{#2}(#1)}
\newcommand{\ctble}{\textbf{ctble}}
\newcommand{\sigmacent}{\sigma\textbf{-cent}}

\newcommand{\cofnwdX}[1]{\mathrm{cof}(\mathrm{nwd}(#1))}
\newcommand{\Vuniverse}{\mathbf{V}}
\newcommand{\Q}{\mathbb{Q}}
\newcommand{\nwd}{\mathfrak{N}}
\newcommand{\Ex}[2][X]{\mathrm{Ex}_{#1}(#2)}
\newcommand{\pair}[1]{\langle #1 \rangle}
\newcommand{\B}{\mathcal{B}}

\renewcommand{\d}{\mathfrak{d}}
\newcommand{\nonM}{\mathrm{non}(\mathcal{M})}
\newcommand{\bairew}{{}\sp{\omega}{\omega}}
\newcommand{\gen}{\mathcal{G}}

\newtheorem{thm}{Theorem}[section]
\newtheorem{lemma}[thm]{Lemma}
\newtheorem{propo}[thm]{Proposition}
\newtheorem{coro}[thm]{Corollary}

\newtheorem{ex}[thm]{Example}
\newtheorem{ques}[thm]{Question}

\title{Remote filters and discretely generated spaces}
\author{Rodrigo Hern\'andez-Guti\'errez}
\address{UAP Cuautitl\'an Izcalli, Universidad Aut\'onoma del Estado de M\'exico. Paseos de las Islas S/N, Atlanta 2da. secci\'on. Cuautitl\'an Izcalli, 54740, M\'exico.  }
\email{rjhernandezg@uaemex.mx}
\date{\today}

\subjclass[2010]{54A25, 54D40, 54D80, 54A35, 03E17}
\keywords{Discretely generated space, one-point compactification, remote point, small uncountable cardinals}

\begin{document}

\begin{abstract}
Alas, Junqueira and Wilson asked whether there is a discretely generated locally compact space whose one point compactification is not discretely generated and gave a consistent example using CH. Their construction uses a remote filter in $\csfo$ with a base of order type $\omega_1$ when ordered modulo compact subsets. In this paper we study the existence and preservation (under forcing extension) of similar types of filters, mainly using small uncountable cardinals. With these results we show that the CH example can be constructed in more general situations.
\end{abstract}

\maketitle

\section{Introduction}

A topological space $X$ is said to be \emph{discretely generated} at a point $p\in X$ if for every $A\subset X$ with $p\in\cl{A}$ there is a discrete subset $D\subset A$ such that $p\in\cl{D}$. Then $X$ is discretely generated if it is discretely generated at each of its points. This notion was first studied in \cite{disc_gen_first}. 

Notice that first countable spaces are discretely generated. Other examples of discretely generated spaces include box products of monotonically normal spaces (\cite[Theorem 26]{disc_gen_box}) and countable products of monotonically normal spaces (\cite[Corollary 2.6]{junq_wil_prods_disc_gen}). Also, a compact dyadic space is discretely generated if and only if it is metrizable (\cite[Theorem 2.1]{disc_gen_box}). Recently, Alas, Junqueira and Wilson have proved the following result. 

\begin{ex}\label{resultch}\cite[Example 2.13]{alas_junq_wil_class_disc_gen}
CH implies that there is a first-countable, locally compact, Hausdorff and discretely generated space with its one-point compactification not discretely generated.
\end{ex}

The construction of Example \ref{resultch} uses a remote point (see the definition below) of $\csfo$ . The existence of remote points (of separable metrizable spaces) in ZFC was a hard problem that was finally solved by van Douwen (\cite{vd51}) and independently by Chae and Smith (\cite{chae-smith}). However, Example \ref{resultch} requires that the remote filter considered has a base of order type $\omega_1$ with respect to inclusion modulo compact sets and the known ZFC constructions do not have this property (in fact, it is consistent that such a filter does not exist by Corollary \ref{inequalities} below). Thus, the previously named authors posed the following general problem.

\begin{ques} (Alas, Junqueira and Wilson, \cite{alas_junq_wil_class_disc_gen})\label{thequestion}
Is there a locally compact Hausdorff discretely generated space with its one point compactification not discretely generated?
\end{ques}

The purpose of this note is to show that an example similar to Example \ref{resultch} can be constructed in some models of the negation of CH. From the topological point of view, we obtain the following result.

\begin{thm}\label{topothm}
If $\Pint=\cofM$, then there is a locally compact, discretely generated Hausdorff space $X$ with its one-point compactification not discretely generated. Moreover, every point of $X$ has character strictly less than $\Pint$.
\end{thm}

The proof of this result will proceed by methods analogous to those of Example \ref{resultch}. Thus, we will be considering the existence of special kinds of remote filters in separable metrizable spaces. In particular we will ask what the minimal character of a remote filter is and construct remote filters with bases that are well-ordered. This naturally gives rise to questions of independence, some of which are not directly related to Question \ref{thequestion} and are more set-theoretical in nature.

The paper will be organized as follows. Sections 2 is introductory; we give the notation, conventions and known results we will be using. The main body of the paper is section 3, which contains proofs of various cardinal inequalities related to remote points and the proof of Theorem \ref{topothm}. Section 4 discusses the question of when a remote filter in $\csfo$ is still remote in a forcing extension. Finally, section 5 contains some final remarks and the questions we were unable to solve.

We remark that Question \ref{thequestion} remains unsettled in ZFC.

\section{Preliminaries}

The Cantor set is the topological product $\csf$ and $\bairew$ is the set of functions from $\omega$ into itself. Recall that ${}\sp{<\omega}A=\bigcup\{{}\sp{n}A:n<\omega\}$ is the set of finite sequences of elements of a set $A$. A space will be called crowded if it contains no isolated points. The space of rational numbers is $\Q$. In a topological space $X$, given $A\subset X$, the closure of $A$ will be denoted by $\cl{A}$ and its boundary by $\bd{A}$. A standard reference for topological concepts is of course \cite{eng}.

For every Tychonoff space $X$, $\beta X$ denotes the \v Cech-Stone compactification of $X$ and $X\sp\ast=\beta X\setminus X$. If $U\subset X$ is open, let $\Ex{U}=\beta X\setminus\cl[\beta X]{X\setminus U}$.

Recall that a point $p\in X\sp\ast$ is said to be \emph{remote} if $p\notin\cl[\beta X]{A}$ for every nowhere dense subset $A$ of $X$. Extending the definition of remote point, we will say that $F\subset\beta X$ is a \emph{remote set of $X$} if $F\cap\cl[\beta X]{N}=\emptyset$ for every nowhere dense subset $N$ of $X$. Let us remark that in this paper all remote sets of $X$ considered will be closed sets in $\beta X$. By Stone's duality, if $X$ is strongly $0$-dimensional then closed subsets of $\beta X$ correspond to filters in the Boolean algebra of clopen subsets of $X$. Thus, we may dually speak of \emph{remote filters} of (clopen sets of) $X$.

It is well-known (and easy to prove) that if $X$ is crowded and metrizable, a subset $F\subset\beta X$ is remote if and only if it is \emph{far} from $X$; that is, if $F\cap\cl[\beta X]{D}=\emptyset$ for every discrete $D\subset X$.

Recall that in any compact Hausdorff space, the character and the pseudocharacter of closed subsets are equal. This motivates the following notation we will use in this note. When $X$ is any Tychonoff space and $F\subset \beta X$ is closed, the character of $F$ (in $\beta X$!), $\chi(F)$ will be defined to be the minimal $\kappa$ such that there is a collection $\U$ of open subsets of $\beta X$ such that $F=\bigcap\{U\in\U:F\subset U\}$ and $|\U|=\kappa$. A base of a filter $\F$ of clopen sets of a $0$-dimensional space $X$ is a collection $\U$ of clopen subsets of $X$ that generates $\F$ and the character of $\F$, $\chi(\F)$ is the minimal cardinality of a base of $\F$; this is clearly consistent. 

If $U$ and $V$ are clopen sets of a space $X$, we define $U\subset\sp\ast V$ to mean that $U\setminus V$ is compact. Notice that this implies that $\cl[\beta X]{U}\cap X\sp\ast\subset\cl[\beta X]{V}\cap X\sp\ast$.

Example \ref{resultch} was constructed by the use of remote filters in $\csfo$. In particular, the authors of that paper implicitly proved the following statement:

\begin{propo}\label{constructionch}
If there is a remote filter of $\csfo$ with a base of order type $\omega_1$ with the relation $\subset\sp\ast$, then there exists a first countable, locally compact and $0$-dimensional space $X$ such that the one-point compactification of $X$ is not discretely generated.
\end{propo}

In fact, this same construction was first carried out by Bella and Simon in order to construct under CH a compact pseudoradial space that is not discretely generated, see \cite[Theorem 7]{bella-simon-disc-gen}.

Let us recall the definition of some small uncountable cardinals that we will use. Denote by $\meager$ the family of all meager subsets of $\bairew$ (with the product topology) and let $[\omega]\sp\omega=\{A\subset\omega:A\textrm{ is infinite}\}$. A family $\G\subset[\omega]\sp\omega$ is centered if for every finite subcollection $G_0,\ldots,G_n\in\G$ there is $G\in\G$ with $G\subset G_0\cap\ldots\cap G_n$. A pseudointersection of $\G\subset[\omega]\sp\omega$ is a set $A\in[\omega]\sp\omega$ such that $A\setminus G$ is finite for all $G\in\G$. If $\pair{P,\lhd}$ is a poset, $G\subset P$ is said to be cofinal in $P$ if for every $p\in P$ there exists $q\in G$ with $p\lhd q$. If $f,g\in\bairew$, then $f\leq\sp\ast g$ means that $\{n\in\omega:g(n)<f(n)\}$ is finite.
$$
\begin{array}{rcl}
\covM & = & \min\big\{|\G|:\G\subset\meager,\bairew=\bigcup\G\big\},\nonumber\\
\cofM & = & \min\big\{|\G|:\G\subset\meager\textrm{ is cofinal in }\pair{\meager,\subset}\}\big\},\nonumber\\
\nonM & = & \min\big\{|M|:M\subset\bairew,M\notin\meager\big\},\nonumber\\
\Pint & = & \min\big\{|\G|:\G\subset[\omega]\sp\omega\textrm{ is centered and has no pseudointersection}\big\},\nonumber\\
\d & = & \min\big\{|\G|:\G\subset\bairew\textrm{ cofinal in }\pair{\bairew,\leq\sp\ast}\big\}.\nonumber
\end{array}
$$

See \cite{vd62} for a topological introduction, \cite{blass} for a recent survey in a set-theoretic perspective and \cite{bart} for information on consistency results. In particular, the following relations are known,
$$
\xymatrix{
& \nonM \ar[r] & \cofM \ar[r] & \C\\
& & \d \ar[u] & \\
\omega_1 \ar[r] & \Pint \ar[r] \ar[uu] & \covM \ar[u] & 
}
$$
where each arrow $\kappa\to\tau$ means that $\kappa\leq\tau$. Further, any of the inequalities is consistently strict.

We will assume the terminology of posets such as dense subsets and filters and so on, see \cite[Section 1.4]{bart}. If $\class$ if a class of partially ordered sets and $\kappa$ is a cardinal number, $\MA[\class]{\kappa}$ is the following assertion:
\begin{quote}
If $\poset\in\class$ and $\{D_\alpha:\alpha<\kappa\}$ are dense subsets of $\poset$, there is a filter $G\subset\poset$ that intersects $D_\alpha$ for every $\alpha<\kappa$.
\end{quote}

Let $\ctble$ be the class of countable posets and $\sigmacent$ the class of $\sigma$-centered posets. 

\begin{thm}\label{martins}
\begin{itemize}
\item[(a)] $\kappa<\Pint$ if and only if $\MA[\sigmacent]{\kappa}$. (\cite[14C, p. 25]{fremlin}) 
\item[(b)] $\kappa<\covM$ if and only if $\MA[\ctble]{\kappa}$. (\cite[Theorem 2.4.5]{bart})
\end{itemize}
\end{thm}

We will use another cardinal invariant defined in \cite{comb_dense_Q}. For every space $X$ we define $\cofnwdX{X}$ to be the smallest $\kappa$ such that there is a family $\G$ of nowhere dense subsets of $X$ such that $|\G|=\kappa$ and cofinal in the poset of nowhere dense subsets of $X$ with respect to inclusion.

\begin{thm}\cite[1.5 and 1.6]{comb_dense_Q}\label{cofnwdrationals}
If $X$ is a crowded separable metrizable space, then $\cofM=\cofnwdX{X}$.
\end{thm}

\section{Character of remote filters}

In order to construct examples for Question \ref{thequestion}, we will first ask what is the character of the smallest remote filter in $\csfo$. For a Tychonoff non-compact space $X$, we define $\remcard(X)$ to be the minimal $\kappa$ such that there exists a remote set $F$ in $X$ closed in $\beta X$ and with $\chi(F)=\kappa$; if no remote filters exist, $\remcard(X)=\infty$. $\remcard$ will denote $\remcard(\Q)$.

\begin{propo}
If $Y$ is a dense subset of a Tychonoff space $X$, then $\remcard(Y)\leq\remcard(X)$
\end{propo}
\begin{proof}
Let $f:\beta Y\to\beta X$ be the unique continuous function such that $f\!\!\restriction_{Y}:Y\to X$ is the inclusion. Let $F\subset\beta X$ be a remote closed subset and let $\U$ be a collection of open subsets of $\beta X$ whose intersection is $F$. Then $G=f\sp\leftarrow[F]$ is a closed subset of $\beta Y$ and $\{f\sp\leftarrow[U]:U\in\U\}$ witnesses that $\chi(G)\leq\chi(F)$. So it remains to show that $G$ is remote.

Assume that $N\subset Y$ is a closed nowhere dense subset of $Y$ and there is $p\in\cl[\beta Y]{N}\cap G$. Then $M=\cl[X]{N}$ is nowhere dense in $X$. From the fact that $f$ is continuous, it follows that $f[\cl[\beta Y]{N}]\subset\cl[\beta X]{f[N]}=\cl[\beta X]{M}$. Thus, $f(p)\in\cl[\beta X]{M}\cap F$, which is a contradiction. Then $G$ is remote and the result follows. 
\end{proof}

Since every crowded, separable and metrizable space has a dense subset homeomorphic to $\Q$, we obtain the following.

\begin{coro}\label{coroQ}
If $X$ is non-compact, crowded, separable and metrizable, then $\remcard\leq\remcard(X)$.
\end{coro}

If $X$ and $Y$ are both non-compact, crowded, separable and metrizable metrizable spaces, then by Theorem \ref{cofnwdrationals}, their ideals of nowhere dense sets behave in a similar fashion. Thus, it is conceivable that in order to define a remote filter, one needs to avoids the same quantity of nowhere dense sets in both cases and then $\remcard(X)=\remcard(Y)$. However, the author was unable to prove this so we will leave it as an open problem, see Question \ref{Qvssepmet} below.

Let us now show how to modify van Douwen's proof of existence of remote points (\cite{vd51}) to prove the existence of remote points with consistently small character. The proof is practically the same modulo some small changes. However, for the reader's convenience and in order to make the changes explicit, we include a full proof.

\begin{propo}
If $X$ is a non-compact, crowded, separable metrizable space, there is a non-empty remote closed subset $F$ of $X$ with $\chi(F)\leq\cofM$.
\end{propo}
\begin{proof}
According to Theorem \ref{cofnwdrationals}, there exists a collection $\nwd$ of $\cofM$ closed nowhere dense subsets of $X$ such that every nowhere dense subset of $X$ is contained in some nowhere dense subset of $\nwd$.

Let $\{I_n:n<\omega\}$ be a discrete family of non-empty open subsets of $X$ and let $\{B_n:n<\omega\}$ be a base of non-empty open subsets of $X$. For each $N\in\nwd$, let 
$$
K(N,n)=\{i<\omega:\cl{B_i}\subset I_n\setminus N\}.
$$
Now, recursively define 
\begin{eqnarray}
k(N,n,0)&=&\min K(N,n)\nonumber\\
k(N,n,m+1)&=&\min\{i<\omega:i\geq k(N,n,m),\textrm{ and for each }s\leq k(N,n,m)\textrm{ with}\nonumber\\
& &\cl{B_s}\subset I_n,\textrm{ there is }t\in K(N,n)\textrm{ with }t\leq i,B_t\subset B_s\}.\nonumber
\end{eqnarray}
And finally, let
$$
U(N,n)=\bigcup\{B_i:i\in K(N,n),i\leq k(N,n,n)\}.
$$
Notice that $U(N,n)$ is a non-empty open set. Since $U(N,n)$ is a finite union of subsets indexed in $K(N,n)$, we obtain that $\cl{U(N,n)}\subset I_n\setminus N$.

Let $U(N)=\bigcup\{U(N,n):n<\omega\}$ for each $N\in\nwd$, notice that $\cl{U(N)}\cap N=\emptyset$. Define $F=\bigcap\{\cl[\beta X]{U(N)}:N\in\nwd\}$.

To prove that $F\neq\emptyset$, by compactness it is enough to show that if $N_0,\ldots,N_{n-1}$ are in $\nwd$ then $U(N_0,n)\cap U(N_1,n)\cap\ldots\cap U(N_{n-1},n)\neq\emptyset$. By rearranging if necessary, we may assume that whenever $r\leq s<n$ we have $k(N_r,n,r)\leq k(N_s,n,r)$. By recursion, define $t(0)=k(N_0,n,0)$ and $t(j+1)=\min\{s\in K(N_{j+1},n):B_s\subset B_{t(j)}\}$ for $j<n$, this is clearly well-defined. By recursion it is easy to check that $t(j)\leq k(D_j,n,j)$ for each $j<n$. Then it follows that $U(N_0,n)\cap U(N_1,n)\cap\ldots\cap U(N_{n-1},n)\supset B_{t(n-1)}$ which is non-emtpy.

By the choice of $\nwd$ it follows that $F$ is a non-empty remote closed set of $X$. We still have to prove that $\chi(F)\leq\covM$. It is enough to argue that 
$$
(\ast)\ F=\bigcap\{\Ex{U(N)}:N\in\nwd\}.
$$

Fix $N\in\nwd$, we know that $F\subset\cl[\beta X]{U(N)}$. According to \cite[Lemma 3.1]{vd51}, $\cl[\beta X]{U(N)}=\cl[\beta X]{\Ex{U(N)}}=\Ex{U(N)}\cup\bd[\beta X]{\Ex{U}}$. Moreover, by \cite[Lemma 3.2]{vd51}, $\bd[\beta X]{\Ex{U}}=\cl[\beta X]{\bd{U}}$. Since $F$ is remote, $F\cap\cl[\beta X]{\bd{U}}=\emptyset$. Thus, $F\subset\Ex{U(N)}$. The other inclusion in $(\ast)$ is clear by the definition.
\end{proof}

We also give lower bounds for the character of remote filters.

\begin{thm}\label{geqcovM}
Let $F$ be a non-empty closed remote subset in a non-compact, crowded, separable metrizable space $X$. Then $\chi(F)\geq\covM$.
\end{thm}
\begin{proof}
Assume that there is a collection $\{U_\alpha:\alpha<\kappa\}$ of open subsets of $\beta X$ with intersection equal to $F$ such that $\kappa<\covM$. We will derive a contradiction.

Since $X$ is Lindel\"of, there is a countable collection $\{V_n:n<\omega\}$ of non-empty open subsets of $X$ such that $F\subset\Ex{V_n}$ and $\cl{V_{n+1}}\subset V_n$ for each $n<\omega$. For each $n<\omega$, let $W_n=V_n\setminus\cl{V_{n+1}}$, then $\{W_n:n<\omega\}$ is a collection of non-empty crowded open subsets of $X$.

For each $n<\omega$, let $\{x(n,m):m<\omega\}$ be a countable dense subset of $W_n$. For each $\alpha<\kappa$, there is an infinite set $E_\alpha\subset\omega$ and a function $f_\alpha:E_\alpha\to\omega$ such that $x(n,f_\alpha(n))\in U_\alpha$ for each $n\in E_\alpha$.

The poset $\poset=({}\sp{<\omega}{\omega},\supset)$ with the inclusion order is countable. For each $\alpha<\kappa$ and $n<\omega$, let
$$
D(\alpha,n)=\{p\in\poset:\exists\ \! m>n\ (m\in{dom}(p)\cap E_\alpha,\ p(m)=f_\alpha(m))\}.
$$
It is not hard to see that $D(\alpha,n)$ is a dense subset of $\poset$. Since $\kappa<\covM$, by Theorem \ref{martins} there exists a filter $G$ that intersects $D(\alpha,n)$ for each $\alpha<\kappa$ and $n<\omega$. Let $f=\bigcup G$. Then it is not hard to show that $f:\omega\to\omega$ is a function such that for each $\alpha<\kappa$, $\{n\in E_\alpha:f(n)=f_\alpha(n)\}$ is non-empty.

Then $D=\{x(n,f(n)):n<\omega\}$ is a closed discrete subset of $X$. So $\cl[\beta X]{D}$ is closed, non-empty and disjoint from $F$ because $F$ is remote. Then there exists $\alpha<\kappa$ such that $U_\alpha\cap \cl[\beta X]{D}=\emptyset$. But then by the definition of $f$, $\{n<\omega:x(n,f(n))\in U_\alpha\}$ is non-empty so $U_\alpha\cap \cl[\beta X]{D}\neq\emptyset$. This is a contradiction so in fact $\chi(F)\geq\covM$.
\end{proof}

\begin{thm}
Let $F$ be a non-empty closed remote subset in a non-compact, crowded, separable metrizable space $X$. Then $\chi(F)\geq\min\{\nonM,\d\}$.
\end{thm}
\begin{proof}
Let $k<\min\{\nonM,\d\}$ and assume that there is a collection $\{U_\alpha:\alpha<\kappa\}$ of open subsets of $\beta X$ with the finite intersection property and whose intersection is equal to $F$. We shall show that there is a nowhere dense subset $N$ of $X$ such that $F\cap\cl[\beta X]{N}\neq\emptyset$. Let us use the following sets defined in the proof of Theorem \ref{geqcovM}:  $\{V_n:n<\omega\}$ and $\{W_n:n<\omega\}$. 

For each $\alpha<\kappa$ and $n<\omega$, choose any point $x(\alpha,n)\in W_n\cap U_\alpha$ whenever this intersection is non-empty. Then for each $n<\omega$ the set $M_n=\{x(\alpha,n):\alpha<\kappa,W_n\cap U_\alpha\neq\emptyset\}$ is meager so there are closed nowhere dense subsets $\{K(n,m):m<\omega\}$ of $W_n$ such that their union contains $M_n$. Without loss of generality, we may assume that $K(n,m)\subset K(n,m+1)$ for each $n,m<\omega$.

For each $\alpha<\kappa$, we define a function $f_\alpha$ with ${dom}(f_\alpha)=\{n<\omega:W_n\cap U_\alpha\neq\emptyset\}$ by $f_\alpha(n)=\min\{k<\omega:x(\alpha,n)\in K(n,k)\}$. Notice that ${dom}(f_\alpha)$ is an infinite set for each $\alpha<\kappa$. By Theorem 3.6 in \cite{vd62}, there is $f:\omega\to\omega$ such that for all $\alpha<\kappa$ and $m<\omega$ there is $n\in{dom}(f_\alpha)\setminus m$ such that $f_\alpha(n)<f(n)$.

Define $N=\bigcup\{K(n,f(n)):n<\omega\}$, this is a closed and nowhere dense subset of $X$. By the properties of $f$, for every $\alpha<\omega$ there is some $n<\omega$ such that $x(\alpha,n)\in N$. So $N$ is a nowhere dense subset of $X$ that intersects every element from $\{U_\alpha:\alpha<\kappa\}$. This shows that $F\cap\cl[\beta X]{N}\neq\emptyset$.
\end{proof}

\begin{coro}\label{inequalities}
If $X$ is a non-compact, crowded, separable metrizable space, then $\covM\leq\remcard(X)\leq\cofM$ and $\min\{\nonM,\d\}\leq\remcard(X)$.
\end{coro}

In order to prove Theorem \ref{topothm}, we will first need to construct a special remote filter.

\begin{thm}\label{filterlinearlyordered}
If $\Pint=\cofM$, then there is a remote filter in $\csfo$ with a base of order type $\Pint$ with respect to the relation $\subset\sp\ast$.
\end{thm}
\begin{proof}
By Theorem \ref{cofnwdrationals}, there is a collection $\{N_\alpha:\alpha<\Pint\}$ of closed nowhere dense subsets of $\csfo$ cofinal in the family of all closed nowhere dense subsets of $\csfo$. Let $\{B(n):n<\omega\}$ be the family of all clopen subsets of $\csf$. By recursion, we will construct a collection of clopen subsets $\{U_\alpha:\alpha<\Pint\}$ of $\csfo$ such that 
\begin{itemize}
\item[(a)] $U_\beta\subset\sp\ast U_\alpha$ if $\alpha<\beta<\Pint$,
\item[(b)] $U_\alpha\cap N_\alpha=\emptyset$ for every $\alpha<\Pint$, and
\item[(c)] given $\alpha<\Pint$, the set of $n<\omega$ such that $U_\alpha\cap(\{n\}\times\csf)\neq\emptyset$ is infinite.
\end{itemize}
Then the filter of clopen subsets of $\csfo$ generated by $\{U_\alpha:\alpha<\Pint\}$ will be as required.

Assume that we have constructed the clopen sets $\{U_\alpha:\alpha<\kappa\}$ for some $\kappa<\Pint$, we would like to choose $U_\kappa$.  For each $\alpha<\kappa$, let $E_\alpha=\{n<\omega: U_\alpha\cap(\{n\}\times\csf)\neq\emptyset\}$ and $f_\alpha:E_\alpha\to\omega$ be a function such that $U_\alpha\cap(\{n\}\times\csf)=\{n\}\times B(f_\alpha(n))$ for all $n\in E_\alpha$. Consider the poset
\begin{align}
\poset=&\{\pair{s,F}:s\in[\omega\times\omega]\sp{<\omega}\textrm{ is a function}, F\in[\kappa]\sp{<\omega},\textrm{ and}\nonumber\\
&\hskip20pt\forall n\in{dom}(s)\ ([\{n\}\times B(f(n))]\cap N_\kappa=\emptyset)\}\nonumber
\end{align}
where $\pair{t,G}\leq\pair{s,F}$ if
\begin{itemize}
\item[(1)] $s\subset t$,
\item[(2)] $F\subset G$ and
\item[(3)] if $\alpha\in F$ and $k\in{dom}(t)\setminus {dom}(s)$, then $k\in E_\alpha$ and $B(t(k))\subset B(f_\alpha(k))$.
\end{itemize}
First, notice that $\poset$ is $\sigma$-centered: if $s\in[\omega]\sp{<\omega}$ is a function, $\{\pair{t,F}\in\poset:t=s\}$ is centered. For $\gamma<\kappa$, let $D(\gamma)=\{\pair{s,F}\in\poset:\gamma\in F\}$ and for $m<\omega$, let $E(m)=\{\pair{s,F}:{dom}(s)\setminus m\neq\emptyset\}$. 

It is easy to prove that $D(\gamma)$ is dense for every $\gamma<\kappa$. So let $m<\omega$, we will prove that $E(m)$ is dense. Given $\pair{s,F}\in\poset$, let $\beta=\max F$. Then there is $k<\omega$ such that $E_{\beta}\setminus k\subset E_{\alpha}$ for all $\alpha\in F$, ${dom}(s)\subset k$ and $k\geq m$. So let $i\in E_{\beta}\setminus k$ and choose $j<\omega$ with $\{i\}\times B(j)\subset U_{\beta}\setminus N_\kappa$. Then $\pair{s\cup\{\pair{i,j}\},F}\in E(m)$ and $\pair{s\cup\{\pair{i,j}\},F}\leq\pair{s,F}$.

Since $\kappa<\Pint$, by Theorem \ref{martins}, there exists a filter $\gen$ that intersects each of these dense sets. Let $f_\kappa=\bigcup\{s:\exists F\ (\pair{s,F}\in\gen)\}$. By genericity it easily follows that ${dom}(f_\kappa)$ is infinite. Define $U_\kappa=\bigcup\{\{n\}\times B(f_\kappa(n)):n\in{dom}(f_\kappa)\}$, this is a clopen subset and clearly conditions (b) and (c) hold.

Now let us show that condition (a) holds for this step of the construction. Given $\alpha<\kappa$, we would like to prove that $U_\kappa\subset\sp\ast U_\alpha$. Let $\pair{s,F}\in \gen\cap D(\alpha)$, so that $\alpha\in F$. Let $m=\max{({dom}(s))}$, we claim that $U_\kappa\cap((\omega\setminus m)\times{\csf})\subset U_\alpha$. Let $k>m$ with $k\in {dom}(f_\kappa)$. Then there is $\pair{t,G}\in\gen$ with $k\in{dom}(s)$ and we may assume that $\pair{t,G}\leq\pair{s,F}$. Then by condition (3) in the definition of the poset, $k\in E_\alpha$ and $B(s(k))\subset B(f_\alpha(k))$. Since $s(k)=f_\kappa(k)$, we obtain that $U_\kappa\cap(\{k\}\times\csf)\subset U_\alpha$. But this was true for all $k>m$ so $U_\kappa\subset\sp\ast U_\alpha$.  

This completes the construction of $\{U_\alpha:\alpha<\Pint\}$ as required, which completes the proof of the Theorem.
\end{proof}

\begin{proof}[{\bf Proof of Theorem \ref{topothm}}]
By Theorem \ref{filterlinearlyordered}, there is a collection $\{U_\alpha:\alpha<\Pint\}$ of clopen subsets of $\csfo$ which generates a remote filter of clopen sets and such that $U_\alpha\subsetneq\sp\ast U_\beta$ whenever $\beta<\alpha<\kappa$. For each $\alpha<\kappa$, let $V_\alpha=\cl[\beta(\csfo)]{U_\alpha}\cap (\csfo)\sp\ast$, which is a clopen subset of $\csfo$. Then $F=\bigcap\{V_\alpha:\alpha<\kappa\}$ is a remote set of $\csfo$, closed in $\beta(\csfo)$. We may assume without loss of generality that $U_0=\csfo$ so $V_0=(\csfo)\sp\ast$.

For each $\alpha<\Pint$, let us define 
$$
F_\alpha=\left\{
\begin{array}{ll}
(\bigcap\{V_\beta:\beta<\alpha\})\setminus V_{\alpha+1},&\textrm{ if $\alpha$ is a limit,}\\
V_\alpha\setminus V_{\alpha+1}, &\textrm{ otherwise.}
\end{array}\right.
$$
Now let $X=(\csfo)\cup\{F\}\cup\{F_\alpha:\alpha<\kappa\}$ with the quotient space topology. Then $X$ is a (Tychonoff) compactification of $\csfo$ with remainder $R=\{F\}\cup\{F_\alpha:\alpha<\kappa\}$. Moreover, the function $f:R\to\kappa+1$ defined by $f(F_\alpha)=\alpha$ when $\alpha<\kappa$ and $f(F)=\kappa$ is a homeomorphism, where $\kappa+1$ is given the ordinal topology.

We now argue that $X$ is discretely generated at all points except at $F$. Clearly, $X$ is discretely generated at each point of $\csfo$ because it is first-countable at each of these points. Also, $F\in\cl[X]{\csfo}$ but there is no discrete subset of $\csfo$ whose closure contains the point $F$ because $F$ is a remote closed subset of $\csfo$. 

Now let $\alpha<\kappa$ and let $A\subset X$ such that $F_\alpha\in\cl[X]{A}$. If $F_\alpha\in\cl[X]{A\cap R}$, then since $R$ is linearly ordered and linearly ordered spaces are discretely generated (see \cite[Corollary 3.12]{disc_gen_first}), there is a discrete subset $D\subset R$ such that $F_\alpha\in\cl[X]{D}$. 

Suppose now that $F_\alpha\in\cl[X]{A\cap(\csfo)}$. Let $B=A\setminus U_{\alpha+1}$ and $Y=[(\csfo)\setminus U_{\alpha+1}]\cup\{F_\beta:\beta\leq\alpha\}$. Then $Y$ is a clopen subset of $X$ that contains $B\cup\{F_\alpha\}$ and $F_\alpha\in\cl[Y]{B}$. Moreover, let $g:\beta B\to Y$ be the unique continuous extension of the inclusion $B\subset Y$. Since $F_\alpha$ has character striclty smaller $\Pint$ in $Y$, $g\sp\leftarrow[F_\alpha]$ has character strictly smaller than $\Pint$ in $\beta B$. By the inequalities of Corollary \ref{inequalities}, $g\sp\leftarrow[F_\alpha]$ is not remote in $\beta B$ so there is a discrete subset $D\subset B$ such that $\cl[\beta B]{D}\cap g\sp\leftarrow[F_\alpha]\neq\emptyset$. Thus, $F_\alpha\in\cl[X]{D}$. This shows that $F_\alpha$ is not remote in $X$ so we have finished the proof.
\end{proof}

\section{Some models}

The objective of this section is to ask when a remote filter (in $\csfo$) is still remote under a forcing extension of the universe. In order to do this, we will quote some known results which show that the corresponding forcing notions preserve remote filters. See \cite{kunen-set-theory-2011} for an introduction on forcing and \cite{bart} for more advanced results. We will denote the ground model by $\Vuniverse$.

For a forcing notion $\poset$, we will say that $\poset$ is \emph{nwd-bounding} if whenever $p\in\poset$ and $\dot{A}$ is a $\poset$-name such that $p\Vdash``\dot{A}\textrm{ is closed and nowhere dense in }\csf\textrm{''}$, then there is a closed and nowhere dense subset $B$ of $\csf$ and $q\leq p$ such that $q\Vdash``\dot{A}\subset B\textrm{''}$.

In great part of the literature, authors have been more concerned about when a forcing notion preserves meager sets. However, for proper forcing notions, preserving meager sets is equivalent to being nwd-bounding. The author of this note could not find an explicit proof but one can easily modify the proof of \cite[Lemma 6.3.21]{bart} to obtain this. Moreover, \cite[Theorem 6.3.22]{bart} can then be translated to the following.

\begin{thm}
The countable support iteration of proper forcing notions that are nwd-bounding is also nwd-bounding.
\end{thm}

A forcing notion $\poset$ has the Sacks property if whenever $p\in\poset$ and $\dot{f}$ is a $\poset$-name such that $p\Vdash``\dot{f}\in{}\sp{\omega}{\Vuniverse}\textrm{''}$, there is $F\in{}\sp{\omega}{\Vuniverse}$ such that $|F(n)|\leq 2\sp{n}$ for all $n<\omega$ and $q\leq p$ such that $q\Vdash``\forall n<\omega\ (f(n)\in F(n))\textrm{''}$. According to Miller (\cite{miller-notes}), Shelah proved that any forcing with the Sacks property is nwd-bounding. The author of this note could not find the proof of this result in the literature. For the sake of completeness we include the sketch of a proof provided by Osvaldo Guzm\'an Gonz\'alez.

\begin{lemma}
Any forcing notion with the Sacks property is nwd-bounding.
\end{lemma}
\begin{proof}
First, we need a combinatorial characterization of nowhere dense subsets of $\csf$ in the spirit of \cite[Theorem 2.2.4]{bart}. For each $p\in\csf$ and $j\in\bairew$ that is strictly increasing and $j(0)=0$, let 
$$
\U(p,j)=\{q\in\csf:\exists n\in\omega\ (p\!\!\restriction_{[j(n),j(n+1))}=q\!\!\restriction_{[j(n),j(n+1))})\},
$$
this is an open dense subset of $\csf$. Moreover, it is not hard to prove that if $N$ is a nowhere dense subset of $\csf$ there exists $p\in\csf$ and a strictly increasing $j\in\bairew$ such that $N\cap \U(p,j)=\emptyset$.

So let $\poset$ be a forcing notion with the Sacks property, let $\dot{x}$ and $\dot{f}$ be names and $p\in\poset$ be such that $p\Vdash``\dot{x}\in\csf,\dot{f}\in\bairew\textrm{ is strictly increasing and }\dot{f}(0)=0\textrm{''}$. We must find $q\in\poset$ with $q\leq p$, $y\in\csf\cap\Vuniverse$ and a strictly increasing $g\in\bairew\cap\Vuniverse$ with $g(0)=0$ such that $q\Vdash``\U(y,g)\subset\U(\dot{x},\dot{f})\textrm{''}$.

Recall that a forcing notion with the Sacks property is $\bairew$-bounding, that is, every function in $\bairew\cap\Vuniverse\sp{\poset}$ is (pointwise) bounded by a function in $\bairew\cap\Vuniverse$ (\cite[Lemma 6.3.38]{bart}). Using this, it is not hard to find a strictly increasing function $f\in\bairew\cap\Vuniverse$ with $f(0)=0$ and $p\sp\prime\in\poset$ such that $p\sp\prime\leq p$ and $p\sp\prime\Vdash``{}\forall n<\omega\ \exists m<\omega\ ([\dot{f}(m),\dot{f}(m+1))\subset[f(n),f(n+1)))\textrm{''}$. Then $p\sp\prime\Vdash``\U(\dot{x},f)\subset\U(\dot{x},\dot{f})\textrm{''}$.

Now let $I[n]=[f(n),f(n+1))$ and let us also write $s(n)=(2\sp0+2\sp1+\ldots+2\sp{n})-1$ for each $n<\omega$. In the generic extension, let $\dot{h}$ be a function with domain $\omega$ such that $\dot{h}(n)=x\!\!\restriction_{I[s(n)]\cup\ldots\cup I[s(n+1)-1]}$ for each $n<\omega$. By the Sacks property, there is a function $H\in\Vuniverse$ with domain $\omega$ and $q\leq p\sp\prime$ such that $q\Vdash``\forall n<\omega\ (\dot{h}(n)\in H(n))\textrm{''}$ and $|H(n)|\leq 2\sp{n}$ for each $n<\omega$. We may assume that $H(n)=\{H(n,0),\ldots,H(n,2\sp{n}-1)\}$ are all functions with domain $I[s(n)]\cup\ldots\cup I[s(n+1)-1]$ to $\{0,1\}$. Define $y\in\csf$ in such a way that when $n<\omega$ and $0\leq k<2\sp{n+1}$ then $y\!\!\restriction_{I[s(n)+k]}=H(n,k)\!\!\restriction_{I[s(n)+k]}$. Also let $g\in\bairew$ be defined by $g(n)=f(s(n))$, clearly $g\in\Vuniverse$, $g$ is strictly increasing and $g(0)=0$. From this it is not hard to see that $q\Vdash``\U(y,g)\subset\U(\dot{x},f)\textrm{''}$ so $q\Vdash``\U(y,g)\subset\U(\dot{x},\dot{f})\textrm{''}$ and we have finished the proof.
\end{proof}

Examples of forcing notions with the Sacks property are Silver forcing (see \cite[3.10, p. 17]{jech-multiple} for the definition) and Sacks forcing itself (see \cite[3.4, p. 15]{jech-multiple}). The poset that adds $\kappa$ (for any $\kappa$) Sacks reals side-by-side (a good introduction is \cite{baum-antimartin}) also has the Sacks property. Miller also proved that ``infinitely equal forcing'' (\cite[Definition 7.4.11]{bart}) is nwd-bounding, for a proof again modify the one given in \cite[Lemma 7.4.14]{bart}.

\begin{thm}\label{filterpreserved}
Assume that $\F$ is a remote filter of clopen sets of $\csfo$. If $\poset$ is any nwd-bounding forcing, then $\F$ is still a remote filter of clopen sets of $\csfo$ in the model obtained by forcing with $\poset$.
\end{thm}
\begin{proof}
Let $p\in\poset$ and $\dot{N}$ a name such that $p\Vdash``\dot{N}\textrm{ is a nowhere dense set of } \Vuniverse\sp{\poset}\textrm{''}$. There is a closed nowhere dense set $M\in\Vuniverse$ and $q\in\poset$ with $q\leq p$ such that $q\Vdash``\dot{N}\subset M\textrm{''}$. Since $\F$ is remote in $\Vuniverse$, there is $V\in\F$ such that $V\cap M=\emptyset$. Then $q\Vdash``\exists U\in\F\ (\dot{N}\cap U=\emptyset)\textrm{''}$.
\end{proof}

\begin{coro}\label{coroforcing}
Assume that $\Vuniverse\models CH$ and let $X$ be the first countable space constructed in Theorem \ref{topothm}. If $\poset$ is a forcing notion that is nwd-bounding, then 
$$
\Vuniverse\sp\poset\models\textrm{ ``the one-point compactification of $X$ is not discretely generated''}.
$$
\end{coro}

Notice that by Corollary \ref{coroforcing} and Example \ref{resultch} we have examples of models where there are spaces as required in Question \ref{thequestion}. Namely we have the Sacks model (both by iteration and side-by-side), the Silver model and forcing with infinitely equal forcing. Notice that however in all of these models the equality $\omega_1=\Pint=\cofM$ holds so it is possible to use Theorem \ref{topothm} to infer the existence of spaces as required in Question \ref{thequestion}.

\begin{ques}
Let $\poset$ be a forcing notion that preserves remote filters. Is it true that $\poset$ is nwd-bounding?
\end{ques}

We also include the following result about the character of remote filters here.

\begin{coro}
It is consistent that there is a remote filter of character $\omega_1$ in $\csfo$ but every point of $(\csfo)\sp\ast$ has character $\C=\omega_2$.
\end{coro}
\begin{proof}
Choose any model of ZFC where $\cofM=\omega_1$ and every ultrafilter in $\omega$ has character $\C=\omega_2$. For example, take a model of $CH$ and take the countable support iteration of Silver forcing with length $\omega_2$ (see \cite[Chapter 22]{halbeisen}). Silver forcing has the Sacks property so there are remote filters of character $\omega_1$ (Theorems \ref{filterlinearlyordered} and \ref{filterpreserved}). Further, Silver forcing adds splitting reals (\cite[Lemma 22.3]{halbeisen}) and from this it easily follows that all ultrafilters in $\omega$ in the generic extension must have character $\omega_2$.

Let $f:\csfo\to\omega$ be the function such that $f(x)=n$ if $x\in\{n\}\times\csf$. Let $\beta f:\beta(\csfo)\to\beta\omega$ be the unique continuous extension. Notice that $\beta f[(\csfo)\sp\ast]=\omega\sp\ast$ because $f$ is continuous and perfect. By \cite[Theorem 2.2]{vd82}, $\beta f$ is open.

Let $p\in(\csfo)\sp\ast$ and assume that $\U$ is a local base of open subsets of $\beta(\csfo)$ at $p$. Then it is not hard to see that $\{\beta f[U]:U\in\U\}$ is a local base of open subsets of $\beta\omega$ at $\beta f(p)$. Since the character of $\beta f(p)$ in $\beta\omega$ is $\C$, then $|\U|\geq\C$. So the character of $p$ in $\beta(\csfo)$ is $\C=\omega_2$.
\end{proof}

\section{Some other remarks and problems}

Before going to the questions, let us explore another cardinal invariant directly related to discretely generated spaces. Let $\dg$ be the smallest cardinal $\kappa$ such that there is a non-discretely generated countable regular space with weight $\kappa$. The proof of Theorem \ref{geqcovM} can be easily modified to show the following.

\begin{thm}
$\covM\leq\dg\leq\remcard$
\end{thm}
\begin{proof}
To prove that $\dg\leq\remcard$, let $F$ be a remote closed set of $\Q$ with $\chi(F)=\remcard$ and consider $\Q\cup \{F\}$ as the subspace of the quotient of  $\beta\Q$ defined by shrinking $F$ to a point.

Now assume that $Q$ is a countable regular space with weight $\kappa<\covM$, we shall prove that $Q$ is discretely generated. Let $A\subset Q$ and $p\in\cl[Q]{A}$, we may assume that $p\notin A$. We may write $A=\bigcup\{A_n:n<\omega\}$ where $A_n$ is clopen in $A$ and $p\notin\cl[A]{A_n}$ for each $n<\omega$. Let $\{U_\alpha:\alpha<\kappa\}$ be a local base at $p$.

For each $n<\omega$, let $A_n=\{x(n,m):m<\omega\}$ be an enumeration and let $\B_n$ be the collection of all open subsets of $A_n$. Also, for each $n,m<\omega$, let $\B_n(m)=\{U\in\B_n:x(n,m)\in U\}$.

Considering the poset $\poset=({}\sp{<\omega}{\omega},\supset)$, we may follow the methods of Theorem \ref{geqcovM} to find a function $g:\omega\to\omega$ such that for every $\alpha<\kappa$ there is $n<\omega$ such that $x(n,g(n))\in U_\alpha$. This proves that $\{x(n,g(n)):n<\omega\}$ is a closed discrete subset of $A$ that has $p$ in its closure. So $Q$ is discretely generated.
\end{proof}

Aside from the original Question \ref{thequestion} that motivated the topic of this paper, we have moved towards studying the cardinal invariant $\remcard$. Also, more generally, we would also like to know more about the structure of remote filters.

\begin{ques}
Is there a remote filter (in $\csfo$) that has a base that is well-ordered in ZFC?
\end{ques}

\begin{ques}\label{Qvssepmet}
If $X$ is a non-compact, crowded, separable and metrizable space, is $\remcard(X)=\remcard$?
\end{ques}

\begin{ques}
Is $\remcard=\cofM$?
\end{ques}

\begin{ques}
Is $\dg=\remcard$?
\end{ques}

\section*{Aknowledgements}
The author would like to thank Professors Angel Tamariz-Mascar\'ua and Richard Wilson for encouraging me to take a look at discretely generated spaces and the Set Theory and Topology group from the Centro de Ciencias Matem\'aticas at Morelia for their comments. Finally, I also thank one of the anonymous referees for detecting an error in a preliminary version of this paper, which lead to the formulation of Question \ref{Qvssepmet}.

\end{document}